\documentclass{paper}

\usepackage{amsthm} 

\usepackage{amsfonts}
\usepackage{amsmath}
\usepackage{amssymb}
\usepackage{makeidx}
\usepackage{epsfig}
\usepackage{graphicx}
\usepackage{color}
\usepackage{cite}
\usepackage{url}

\usepackage{url}

\usepackage[rgb,letterpaper]{xcolor}
\usepackage{cite}
\usepackage[letterpaper]{psfrag}
\usepackage{tikz}
\usepackage[process=auto,crop=pdfcrop]{pstool}
\newtheorem{theorem}{Theorem}

\newtheorem{lemma}{Lemma}
\newtheorem{corollary}{Corollary}
\newtheorem{definition}{Definition}

\newcommand{\re}{{\mathbb R}}

\newcommand{\n}{{\mathbb N}}
\newcommand{\q}{{\mathbb Q}}

\newcommand{\cM}{{\cal{M}}}
\newcommand{\cm}{{\cal{M}}}

\providecommand{\com[2]}{\begin{tt}[#1: #2]\end{tt}}

\providecommand{\rj[1]}{{\color{black}#1}}

\providecommand{\aaa[1]}{{\color{black}#1}}

\title{{Lower Bounds on Complexity of Lyapunov Functions for Switched Linear Systems}} 

\author{Amir Ali Ahmadi\thanks{Amir Ali Ahmadi is with the Department of Operations Research and Financial Engineering at Princeton University. Email: \texttt{a\_a\_a@princeton.edu}. He has been partially supported for this work by the AFOSR Young Investigator Program Award.} and Rapha\"el~M. Jungers\thanks{Rapha\"el~M. Jungers is with the ICTEAM Institute, Universit\'e catholique de Louvain. Email: \texttt{raphael.jungers@uclouvain.be}. He has been partially supported for this work by the Communaut\'e francaise de Belgique - Actions de Recherche Concert\'ees, and by the Belgian Programme on Interuniversity Attraction Poles initiated by the Belgian Federal Science Policy Office. R.J. is a F.R.S.-FNRS Research Associate. }
}

\begin{document}
\maketitle


%



%
\begin{abstract}
 We show that for any positive integer $d$, there are families of switched linear systems---in fixed dimension and defined by two matrices only---that are stable under arbitrary switching but do not admit (i) a polynomial Lyapunov function of degree $\leq d$, or (ii) a polytopic Lyapunov function with $\leq d$ facets, or (iii) a piecewise quadratic Lyapunov function with $\leq d$ pieces. This implies that there cannot be an upper bound on the size of the linear and semidefinite programs that search for such stability certificates.\\ Several constructive and non-constructive arguments are presented which connect our problem to known (and rather classical) results in the literature regarding the finiteness conjecture, undecidability, and non-algebraicity of the joint spectral radius. 
 \\In particular, we show that existence of \rj{an extremal piecewise algebraic Lyapunov function implies the finiteness property of the optimal product, generalizing a result of Lagarias and Wang.  \aaa{As a corollary, we prove that the finiteness property holds for sets of matrices with an extremal Lyapunov function belonging to some of the most popular function classes in controls.}} 
\end{abstract}

%
%

\section{Introduction}
\aaa{

In recent years, hybrid systems have been well recognized as a powerful modeling framework for complex dynamical systems. This in turn has led to a surge of activity to design computational tools for their analysis and control. Many of these tools are developed based on a synergy between classical ideas from Lyapunov theory (often appropriately modified or generalized to deal with the intricacies of hybrid systems) and modern techniques of numerical optimization, notably convex optimization. The outcome of this synergy is an automated search process for appropriately-defined Lyapunov functions that certify desired properties of system trajectories (e.g., stability, boundedness, etc.). 

Our goal in this paper is to show some basic and concrete limitations of this broad class of techniques for analysis of hybrid systems. Since we are aiming at negative results, the simpler we can make the mathematical setting of our study, the more powerful the implications of our theorems will be. For this reason, we choose to focus on arguably the simplest class of hybrid systems, namely, \emph{switched linear systems}, and ask arguably the simplest question of interest about them, namely, \emph{stability}.
}

\rj{More precisely,} the input to our problem is a set of $m$ real $n\times n$ matrices $\Sigma\mathrel{\mathop:}=\{A_1,\ldots,A_m\}$. This set describes a switched linear system of the form
\begin{equation}\label{eq:switched.linear.sys}
x_{k+1}=A_{\sigma\left(  k\right)  }x_{k},
\end{equation}
where $k$ is the index of time and $\sigma:\mathbb{Z\rightarrow}\left\{  1,...,m\right\}$ is a
map from the set of positive integers to the set $\{1,\ldots,m\}$. A basic notion of stability is that of \emph{absolutely asymptotically
stable} (AAS), also referred to \emph{asymptotic stability under arbitrary switching} (ASUAS), which asks whether all initial conditions in $\mathbb{R}^n$ converge to the origin for \emph{all} possible switching sequences. It is not difficult to show that absolute asymptotic
stability of (\ref{eq:switched.linear.sys}) is equivalent to absolute asymptotic stability of the linear difference inclusion
\begin{equation}\label{eq:linear.difference.inclusion}
x_{k+1}\in \mbox{co}{\Sigma}\ x_{k},
\end{equation}
where $\mbox{co}{\Sigma}$ here denotes the convex hull of the
set $\Sigma$. \rj{Dynamical systems of this type have been widely studied since they model a linear system which is subject to time-dependent uncertainty}. See for instance \cite{liberzon-switching}, \cite{shorten-siamreview07}, or \cite{Raphael_Book} for applications in systems and control.

When the set $\Sigma$ consists of a single matrix $A$ (i.e., $m=1$), we are of course in the simple case of a linear system where asymptotic stability is equivalent to the spectral radius of $A$ having modulus less than one. This condition is also equivalent to existence of a \emph{quadratic} Lyapunov function and can be checked in polynomial time. When $m\geq 2$, however, no efficiently checkable criterion is known for AAS. Arguably, the most promising approaches in the literature have been to use convex optimization (typically linear programming (LP) or semidefinite programming (SDP)) to construct Lyapunov functions that serve as certificates of stability. The most basic example is that of a \emph{common quadratic Lyapunov function} (CQLF), which is a positive definite quadratic form $x^TQx$ that decreases with respect to all $m$ matrices, i.e., satisfies $$x^T(A_i^TQA_i-Q)x<0, \forall x\in\mathbb{R}^n\aaa{\setminus\{0\}}, i=1,\ldots,m.$$ On the positive side, the search for such a quadratic function is efficient numerically as it readily provides a semidefinite program. On the negative side, and in contrast to the case of linear systems, existence of a CQLF is a sufficient but \emph{not} necessary condition for stability. Indeed, a number of authors have constructed examples of AAS switched systems which do not admit a CQLF and studied various criteria for existence of a CQLF (\cite{Ando98,Dayawansa_switching,mason-shorten,Olshevsky_no_quadratic}).

To remedy this shortcoming, several richer and more complex classes of Lyapunov functions have been introduced.  We list here the five that are perhaps the most ubiquitous:

{\bf Polynomial Lyapunov functions.} A homogeneous\footnote{Since the dynamics in (\ref{eq:switched.linear.sys}) is homogeneous, there is no loss of generality in parameterizing our Lyapunov functions as homogeneous functions. Also, we drop the prefix ``common'' from the terminology ``common polynomial Lyapunov function'' as it is implicit that our Lyapunov functions are always common to all $m$ matrices $A_i$ in $\Sigma$.} multivariate polynomial $p(x)$ of some even degree $d$ is a \emph{polynomial Lyapunov function} for (\ref{eq:switched.linear.sys}) if it is positive definite\footnote{A form (i.e., homogeneous polynomial) $p$ is \emph{positive definite} if $p(x)>0$ for all $x\neq 0$.} and makes $p(x)-p(A_ix)$ positive definite for $i=1,\ldots,m$.

Although this is a rich class of functions, a numerical search for polynomial Lyapunov functions is an intractable task even when the degree $d$ is fixed to $4$. In fact, even testing if a given quartic form is positive definite is NP-hard in the strong sense (see, e.g.,\cite{AAA_Cubic_vec_field}). A popular and more tractable subclass of polynomial Lyapunov functions is that of sum of squares Lyapunov functions.

{\bf Sum of squares (sos) Lyapunov functions.} A homogeneous polynomial of some even degree $d$ is a \emph{sum of squares (sos) Lyapunov function} for (\ref{eq:switched.linear.sys}) if $p$ is positive definite and a sum of squares\footnote{A polynomial $p$ is a sum of squares if it can be written as $p=\sum_i q_i^2$ for some polynomials $q_i$.}, and if all the $m$ polynomials $p(x)-p(A_ix)$, $i=1,\ldots,m$ are also positive definite and sums of squares.

For any fixed degree $d$, the search for an sos Lyapunov function of degree $d$ is a semidefinite program of size polynomial in the input. When $d=2$, this class coincides with CQLFs as nonnegative quadratic forms are always sums of squares. Moreover, existence of an sos Lyapunov function is not only sufficient but also necessary for AAS of (\ref{eq:switched.linear.sys})~\cite{Pablo_Jadbabaie_JSR_journal}. This of course implies that existence of polynomial Lyapunov functions defined above is also necessary and sufficient for stability.

{\bf Polytopic Lyapunov functions.} A polytopic Lyapunov function $V$ for (\ref{eq:switched.linear.sys}) with $d$ pieces is one that is a pointwise maximum of $d$ linear functions: $$V(x)\mathrel{\mathop:}=\max_{i=1\ldots,d} |c_i^Tx|, $$ where $c_1,\ldots,c_m$ span $\mathbb{R}^n$.  The sublevel sets of such functions are polytopes, justifying their name. Polytopic Lyapunov functions (with enough number of pieces) are also necessary and sufficient for absolute asymptotic stability. One can use linear programming to search for subclasses of these Lyapunov functions. These subclasses are big enough to also comprise a necessary and sufficient condition for stability (see \cite{polanski1997LyapunovLP,blanchini-miani,polanski2000absolute,switched_system_survey}).

{\bf Max-of-quadratics Lyapunov functions.}  A \emph{max-of-quadratics} Lyapunov function $V$ for (\ref{eq:switched.linear.sys}) with $d$ pieces is one that is a pointwise maximum of $d$ positive definite quadratics: $$V(x)\mathrel{\mathop:}=\max_{i=1,\ldots,d} x^TQ_ix, $$ where $Q_i\succ 0$.  The sublevel sets of such functions are intersections of ellipsoids.

{\bf Min-of-quadratics Lyapunov functions.}  A \emph{min-of-quadratics} Lyapunov function $V$ for (\ref{eq:switched.linear.sys}) with $d$ pieces is one that is a pointwise minimum of $d$ positive definite quadratics: $$V(x)\mathrel{\mathop:}=\min_{i=1,\ldots,d} x^TQ_ix, $$ where $Q_i\succ 0$.  The sublevel sets of such functions are unions of ellipsoids. 

By a {\bf piecewise quadratic Lyapunov function}, we mean one that is either a max-of-quadratics or a min-of-quadratics. Both of these families are known to provide necessary and sufficient conditions for AAS. Several references in the literature produce semidefinite programs that can search over a subclass of max-of-quadratics or min-of-quadratics Lyapunov functions (see \cite{convex_conjugate_Lyap}). These subclasses alone also provide necessary and sufficient conditions for AAS.
A unified framework to produce such SDPs is presented in \cite{JSR_path.complete_journal},~\cite{HSCC_JSR_Path_complete}, where a recipe for writing down stability proving linear matrix inequalities is presented based on some connections to automata theory.

For all classes of functions we presented, one can think of $d$ as a \emph{complexity parameter} of the Lyapunov functions. The larger the parameter $d$, the more complex our Lyapunov function would look like and the bigger the size of an LP or an SDP searching for it would need to be.

\subsection{Motivation and contributions}
Despite the encouraging fact that all five classes of Lyapunov functions mentioned above provide necessary and sufficient conditions for AAS of (\ref{eq:switched.linear.sys}) that are amenable to computational search via LP or SDP, all methods offer an \emph{infinite hierarchy} of algorithms, for increasing values of $d$, leaving unclear the natural questions: How high should one go in the hierarchy to obtain a proof of stability? How does this number depend on $n$ (the dimension) and $m$ (the number of matrices)? Unlike the case of CQLF which is ruled out as a necessary condition for stability through several counterexamples in the literature, we are not aware of that many counterexamples that rule out more complicated Lyapunov functions. For example, is there an example of a set of matrices that is AAS but does not admit a polynomial Lyapunov function of degree $4$, or $6$, or $200$?\footnote{The largest degree existing counterexample that we know of is one of our own, appearing in \cite{sosconvex_Lyap_cdc}, which is a pair of AAS $2\times 2$ matrices with no polynomial Lyapunov function of degree $14$.} Or, is there an example of a set of matrices that is AAS but does not admit a piecewise quadratic Lyapunov function with $200$ pieces? If such sets of matrices exist, how complicated do they look like? How many matrices should they have and in what dimensions should they appear?

In this paper we give an answer to these questions, providing constructive and non-constructive arguments for existence of ``families of \rj{very} bad matrices'', i.e., those forcing the complexity parameter $d$ of all Lyapunov functions to be arbitrarily large, even for fixed $n$ and $m$ (in fact, even for the minimal \rj{nontrivial} situation $n=m=2$). The formal statement is given in Theorem~\ref{thm:Lyaps.blow.up.families} below.

It is important to remark that the families of matrices we present have already appeared in rather well-established literature, though for different purposes. These matrices have to do with the ``non-algebraicity'' and the ``finiteness property'' of the notion of \emph{joint spectral radius} (JSR) (see Sections \ref{section-algebraicity} and \ref{section-finiteness} for definitions). This leaves us with the much simplified task of establishing a formal connection between these two concepts and that of complexity of Lyapunov functions. \aaa{We hope that clarifying these connections sheds new light on the intrinsic relationship between the JSR and the stability question for switched linear systems.
Indeed, many of the results that we refer to in the literature on the JSR appear much before counterexamples to existence of CQLF in the switched systems literature.


The theorem stated below is proven over the next two sections.}

\begin{theorem} \label{thm:Lyaps.blow.up.families}
For any positive integer $d$, the following families of matrices (parameterized by $k$) include switched systems that are asymptotically stable under arbitrary switching but do not admit (i) a polynomial (hence sos) Lyapunov function of degree $\leq d$, or (ii) a polytopic Lyapunov function with $\leq d$ facets, or (iii) a max-of-quadratics Lyapunov function with $\leq d$ pieces, or (iv) a min-of-quadratics Lyapunov function with $\leq d$ pieces:
\begin{enumerate}
\item $(1-\frac{1}{k})\{A_1,A_2\}$, with
 $$A_1=\frac{(1-t^4)}{(1-3\pi t^3/2)}\begin{bmatrix} {\sqrt{1-t^2}} & -{t} \\ 0 & 0\end{bmatrix},$$ $$A_2=(1-t^4)\begin{bmatrix} \sqrt{1-t^2} & -t \\ t & \sqrt{1-t^2}\end{bmatrix},$$ 
where $t=\sin\frac{2\pi}{2k+1}$ and $k=1,2,\ldots$.

\noindent (This family appears in the work of \cite{Koz90} as an example demonstrating that the joint spectral radius is not a semialgebraic quantity; see Section \ref{section-algebraicity}.)

\item  $$A_1=\begin{bmatrix} 1 & 1 \\ 0 & 1\end{bmatrix}, A_2=\alpha\begin{bmatrix} 1 & 0 \\ 1 & 1\end{bmatrix},$$ with $$\alpha<\alpha_*\simeq 0.749326546330367557943961948091344672091....$$ \noindent (This set appears in \cite{finiteness_conjec_counterexample,FP_explicit_counterexample} as a counterexample to the finiteness property.)

\item $(1-\frac{1}{k})\{A_1,A_2\}$, with 
$$A_1=\alpha^k\begin{bmatrix} 0 & 0 \\ 1 & 0\end{bmatrix},\ A_2=\alpha^{-1}\begin{bmatrix}
\cos\frac{\pi}{2k} & \sin\frac{\pi}{2k}  \\ -\sin\frac{\pi}{2k} &\cos\frac{\pi}{2k}\end{bmatrix},$$ where $$1<\alpha<(\cos\frac{\pi}{2k})^{-1}.$$ (This family appears in the work of \cite{LagariasWang_Finiteness_Conjecture} as an example demonstrating that the length of the optimal product cannot be bounded; see Section \ref{section-finiteness}.)
 
\end{enumerate}

\end{theorem}

The first construction and its relation to non-algebraicity is presented in Section \ref{section-algebraicity}. The second and third constructions are similar and their relations to the finiteness property are presented in Section \ref{section-finiteness}. 
In Section \ref{section-finiteness}, we present a result that is of potential interest independent of the above theorem: \aaa{that existence of an extremal \emph{piecewise algebraic} Lyapunov function implies the finiteness property of optimal products. Piecewise algebraic Lyapunov functions constitute a broad class of functions that includes all classes stated so far as a special case. By an \emph{extremal Lyapunov function}, we mean a Lyapunov function $V$ such that $$\sup_{x \in \re^n\setminus\{0\}, i =1,\dots,m}{ \frac{V(A_ix)}{V(x)}}= \rho,$$ where $\rho$ is the JSR.  See \cite[Section 2.1]{Raphael_Book} for more on the similar concept of extremal norms.}



\aaa{In a sense, this result links} lower and upper bound approaches for computation of the joint spectral radius. Similar results were obtained in the pioneering works of \cite{gu} for polytopic Lyapunov functions and \cite{LagariasWang_Finiteness_Conjecture} for quadratic Lyapunov functions, as well as, several other classes of \emph{convex} Lyapunov functions.

We shall also remark that for \emph{continuous time} switched linear systems, \cite{switch.common.poly.Lyap} have established that the degree of a polynomial Lyapunov function for an ASS system may be arbitrarily high, answering a question raised by Dayawansa and Martin. We have been unable to come up with a transformation from continuous time to discrete time that preserves both AAS and non-existence of polynomial Lyapunov functions of any desired degree.

In Section \ref{section-undecidability}, we provide an alternative proof of Theorem \ref{thm:Lyaps.blow.up.families} based on an undecidability results due to \cite{BlTi2}. While this will be a non-constructive argument, its implications will be stronger. Indeed, Theorem \ref{thm:Lyaps.blow.up.families} above implies that the complexity parameter $d$ (and hence the size of \rj{the} underlying LPs and SDPs) cannot be upper bounded as a function of $n$ and $m$ only. The undecidability results, however, imply that $d$ cannot be upper bounded even as a \rj{(computable)} function of $n$, $m$, and the entries of the input matrices.

\aaa{In Section~\ref{sec:numerical.example}, we numerically solve semidefinite programs to demonstrate the result of Theorem~\ref{thm:Lyaps.blow.up.families} with some concrete numbers. It is shown that as some parameter in the input matrices changes, the minimum required degree of a polynomial Lyapunov function indeed grows rapidly (all the way to 24 before we stop the experiment). We close our paper with some brief concluding remarks in Section~\ref{sec:conclusion}.}

\section{Complexity of Lyapunov functions and non-algebraicity}\label{section-algebraicity}



 One classical approach to demonstrate that a problem is hard is to establish that there is no \emph{algebraic criterion} for testing the property under consideration. This is formalized by showing that the set of instances of a given size that satisfy the property do not form a \emph{semialgebraic set} (see formal definition below). Such a result rules out the possibility of any characterization of the property at hand that only involves operations on the input data that include combinations of arithmetical operations (additions, subtractions, multiplications, and divisions), logical operations (``and'' and ``or''), and sign test operations (equal to, greater than, greater than or equal to,...); see~\cite{Blondel_simultaneous_3plants}. While this is a very strong statement, non-algebraicity does not imply (but is implied by) Turing undecidability, which will be our focus in Section~\ref{section-undecidability}. 
%
%
%
%
%
Nevertheless, non-algebraicity results alone are enough to show that the complexity of commonly used Lyapunov functions for switched linear systems cannot be bounded. The goal of this section is to formalize this argument.


\begin{definition}
A set $S\subset\mathbb{R}^n$ defined as $S=\{x\in\mathbb{R}^n:\ f_i(x)\rhd_i 0, i=1,\ldots,r\}$, where for each $i$, $f_i$ is a polynomial and $\rhd_i$ is one of $\geq,<,=,\neq$, is called a \emph{basic semialgebraic set}. A set is called \emph{semialgebraic} if it can be expressed as a finite union of basic semialgebraic sets.
\end{definition}

\begin{theorem}[\cite{Tarski_quantifier_elim}, \cite{Seidenberg_quantifier_elim}] \label{thm:Tarski.Seidenberg}
Let $S\subset\mathbb{R}^{k+n}$ be a semialgebraic set and $\pi: \mathbb{R}^{k+n}\rightarrow\mathbb{R}^n$ be a projection map that sends $(x,y)\mapsto x$. Then $\pi(S)$ is a semialgebraic set in $\mathbb{R}^n$.
\end{theorem}

We start by presenting two examples of semialgebraic sets that are relevant for our purposes.

\begin{lemma}\label{lem:stable.matrices.semialgebraic}
The set $\mathfrak{S}_n$ of stable $n \times n$ real matrices (i.e., those with spectral radius less than one), when viewed as a subset of $\mathbb{R}^{n^2}$, is semialgebraic.
\end{lemma}

\begin{proof}
An equivalent characterization of stable matrices is via quadratic Lyapunov functions: 
\begin{equation}\label{eq:set.of.stable.matrices}
\begin{array}{llll}
\mathfrak{S}_n&=&\{A:&  \exists P=P^T \ \mbox{s.t.}\\ \ &\ &\ &  x^TPx>0, \forall x\neq 0, \\ \ &\ & \ & x^T(P-A^TPA)x>0, \forall x\neq 0 \}.
\end{array}
\end{equation} 
Let $\mathfrak{T}_n$ be the set of $n\times n$ matrices $A$ and $P$ that satisfy these constraints; i.e., $P$ and $P-A^TPA$ both positive definite. The set $\mathfrak{T}_n$ is semialgebraic (in fact basic semialgebraic). One way to see this is to note that the variables $x$ in (\ref{eq:set.of.stable.matrices}) can be eliminated: Since a matrix is positive definite if and only if its $n$ leading principal minors are positive, we can describe the set $\mathfrak{T}_n$ by $2n$ strict polynomial inequalities (in variables that are the entries of $P$ and $A$). The set $\mathfrak{S}_n$ is then the projection of $\mathfrak{T}_n$ onto the space of the $A$ variables. Hence, by Theorem~\ref{thm:Tarski.Seidenberg}, $\mathfrak{S}_n$ is semialgebraic.
\end{proof}

\begin{lemma}\label{lem:nonnegative.polys.semialgebraic}
The set $\mathcal{P}_{n,d}$ of nonnegative polynomials in $n$ variables and (even) degree $d$ is semialgebraic.
\end{lemma}

\begin{proof}
This is a standard fact in algebra; see e.g.~\cite{PabloGregRekha_BOOK}. A polynomial $p$ is by definition nonnegative if $$\forall x_1\forall x_2\ldots\forall x_n\ p(x_1,\ldots,x_n)\geq 0.$$ One can apply quantifier elimination techniques (see e.g.~\cite{QE_book}) to eliminate the quantified variables $(x_1,\ldots,x_n)$ and obtain a description of $\mathcal{P}_{n,d}$ as a semialgebraic set in terms of the coefficients of $p$ only. Another approach is to resort to the representation of nonnegative polynomials as sums of squares of rational functions (\`a la Hilbert's 17th problem~\cite{Reznick}) and note the equivalent description:

\begin{equation}\nonumber
\begin{array}{llll}
\mathcal{P}_{n,d}&=&\{p:&  \exists q\in\mathbb{R}[x] \    \mbox{of degree}\  \hat{d}=\hat{d}(n,d)\ \mbox{s.t.} \\ \ &\ &\ & q \ \mbox{is sos,} \\ \ &\ & \ &  pq \ \mbox{is sos} \}.
\end{array}
\end{equation}  By definition of the sum of squares decomposition, the last two conditions can be written as polynomial equations in the coefficients of the polynomials $p$, $q$, and the coefficients of the new polynomials that are being squared. Hence, the result is a semialgebraic set. The set $\mathcal{P}_{n,d}$ is the projection of this set onto the space of coefficients of $p$ and is therefore semialgebraic.
%
\end{proof}

Unlike the case of stable matrices (Lemma~\ref{lem:stable.matrices.semialgebraic}), when we move to switched systems defined by even only two matrices, the set of stable systems no longer defines a semialgebraic set. This is a result of \cite{Koz90}.  The result is stated in terms of the \emph{joint spectral radius} (see \cite{Raphael_Book} for a monograph on the topic) which captures the stability of a linear switched system\footnote{\rj{See \cite{jsr-toolbox} for a recent software toolbox for computation/approximation of the Joint Spectral Radius.}}.
	\begin{definition}(\cite{RoSt60})\label{def:jsr} If
$\|\cdot\|$ is any matrix norm, \rj{for any integer $k,$ consider 
$$\rho_k(\Sigma):=\sup_{A_i\in \Sigma}\{ \|A_1\dots A_k\|^{1/k}\}.$$ }The {\em joint spectral radius} (JSR) of $\cM$ is
\begin{equation}\label{eq:JSR}
\rho(\Sigma)=\lim_{k\rightarrow \infty} \
\rho_k(\Sigma).
\end{equation}
\end{definition}
The joint spectral radius does not depend on the matrix norm chosen
thanks to the equivalence between matrix norms in finite dimensional
spaces. It is well known that the switched system in (\ref{eq:switched.linear.sys}) is absolutely asymptotically stable if and only if $\rho(\Sigma)< 1.$

\begin{theorem}[\cite{Koz90}; see also~\cite{Theys_thesis}]\label{thm:Kozyakin}
The\\ set of $2\times 2$ matrices $A_1,A_2$ with $\rho(A_1,A_2)<1$ is not semialgebraic. 
\end{theorem} 

\begin{proof}
The proof of Kozyakin is established by showing that the family of matrices
 $$A_1=\frac{(1-t^4)}{(1-3\pi t^3/2)}\begin{bmatrix} {\sqrt{1-t^2}} & -{t} \\ 0 & 0\end{bmatrix},$$ $$A_2=(1-t^4)\begin{bmatrix} \sqrt{1-t^2} & -t \\ t & \sqrt{1-t^2}\end{bmatrix},$$ have JSR less than one for  $t=\sin\frac{2\pi}{2k+1}$ (for $k\in \n$ large enough), and JSR more than one for $t=\sin\frac{2\pi}{2k}$ (for $k\in \n$ large enough). Hence, the stability set has an infinite number of disconnected components and therefore cannot be semialgebraic. 
\end{proof}

We now show that by contrast, for any integer $d$, the set of matrices $\{A_1,\ldots,A_m \}$ that admit a common polynomial Lyapunov function of degree $\leq d$ is in fact semialgebraic. This establishes the result related to the first construction in Theorem~\ref{thm:Lyaps.blow.up.families}.

\begin{theorem}\label{thm:matrix.families.with.Lyaps.are.semialgebraic}
For any positive integer $d$, the set of matrices $\{A_1,\ldots,A_m\}$ (viewed as a subset of $\mathbb{R}^{mn^2}$) that admit either (i) a polynomial Lyapunov function of degree $\leq d$, or (ii) a polytopic Lyapunov function with $\leq d$ facets, or (iii) a piecewise quadratic Lyapunov function (in form of max-of-quadratics or min-of-quadratics) with $\leq d$ pieces is semialgebraic.
\end{theorem}

\begin{proof}
We prove the claim only for polynomial Lyapunov functions. The proof of the other claims are very similar.  
The goal is to show that the set
\begin{equation}\label{eq:set.of.matrices.admitting.common.Lyap.degree.less.d}
\begin{array}{llll}
\mathfrak{T}_{n,m}&\mathrel{\mathop:}=&\{A_1,\ldots,A_m:&  \exists p\mathrel{\mathop:}=p(x),\ degree(p)\leq d, \ \mbox{s.t.}\\ \ &\ &\ &  p(x)>0, \forall x\neq 0, \\ \ &\ & \ & p(x)-p(A_ix)>0, \forall x\neq 0, \\ \ &\ & \ & \forall i=1,\ldots,m \}
\end{array}
\end{equation} 
is semialgebraic. 
Our argument is a simple generalization of that of Lemma~\ref{lem:stable.matrices.semialgebraic} and relies on Lemma~\ref{lem:nonnegative.polys.semialgebraic}. It follows from the latter lemma that for each $\hat{d}\in\{2,\ldots,d\}$, the set of polynomials $p$ of degree $\hat{d}$ and matrices $\{A_1,\ldots,A_m\}$ that together satisfy the constraints in (\ref{eq:set.of.matrices.admitting.common.Lyap.degree.less.d}) is semialgebraic. (Note that these are sets in the coefficients of $p$ and the entries of the matrices $A_i$, as the variables $x$ are being eliminated.) The union of these sets is of course also semialgebraic. By the Tarski--Seidenberg theorem (Theorem~\ref{thm:Tarski.Seidenberg}), once we project the union set onto the space of variables in $A_i$, we still obtain a semialgebraic set. Hence, $\mathfrak{T}_{n,m}$ is semialgebraic.
\end{proof}

\section{Complexity of Lyapunov functions and the finiteness property of optimal products}\label{section-finiteness}

A set of matrices $\{A_1,\ldots,A_m\}$ satisfies the \emph{finiteness property} if its JSR is achieved as the spectral radius of a finite product; i.e., if $$\rho(A_1,\ldots,A_m)=\rho^{1/k}(A_{\sigma_k}\ldots A_{\sigma_1}),$$ for some $k$ and some $(\sigma_k,\ldots,\sigma_1)\in\{1,\ldots,m\}^k$. The matrix product $A_{\sigma_k}\ldots A_{\sigma_1}$ that achieves the JSR is called the \emph{optimal product} and generates the ``worst case trajectory'' of the switched system in (\ref{eq:switched.linear.sys}). The finiteness conjecture of \cite{LagariasWang_Finiteness_Conjecture} (see also \cite{Gur95}, where the conjecture is attributed to Pyatnitskii) asserts that all sets of matrices have the finiteness property. The conjecture was disproved in 2002 by \cite{FP_disproved_Tetris} with alternative proofs consequently appearing in~\cite{finiteness_conjec_counterexample},~\cite{Kozyakin_dynamical_counterexample_FP}, and~\cite{FP_explicit_counterexample}. In particular, the last reference provided the first explicit counterexample only recently. It is currently not known whether all sets of matrices with \emph{rational} entries satisfy the finiteness property~(\cite{jungers-blondel-finiteness}).

\rj{Gurvits has shown} \cite{Gur95} that if the set of matrices admits a polytopic Lyapunov function, then the finiteness property holds. The result \rj{has been generalized by Lagarias and Wang} \cite{LagariasWang_Finiteness_Conjecture} to Lyapunov functions that take the form of various other norms, including ellipsoidal norms. \aaa{For example, their result has the following implication for quadratic Lyapunov functions.}

\begin{theorem}[\cite{LagariasWang_Finiteness_Conjecture}] \label{thm:ellipse.then.finite}
The finiteness property holds for any set of $n\times n$ matrices $\{A_1,\ldots,A_m\}$ {of JSR equal to one} that share an ellipsoidal norm, i.e., satisfy $A_i^TPA_i\preceq P$ for some symmetric positive definite matrix $P$. Moreover, the length of the optimal product is upper bounded by a quantity that depends on $n$ and $m$ only.
\end{theorem}

In this section, we \aaa{further generalize this result to ``piecewise algebraic'' Lyapunov functions (see definition below).  This turns out to include all the classes of Lyapunov functions treated in this paper. Note that polynomial Lyapunov functions of degree $\geq 4$ do not in general define a norm as their sublevel sets may very well be non-convex. This is the reason for the need to work with a more general version of the result of Lagarias and Wang.}

\aaa{A minor technical remark before we get to our generalization: In Theorem~\ref{thm:ellipse.then.finite} (and in Theorem~\ref{thm:sos.Lyap.then.finite} below), the assumption is that the JSR is equal to one (i.e., the system is neither exponentially stable nor exponentially unstable).  However, the result trivially applies to any positive value of the JSR, by homogeneity of this quantity.  The condition would instead become that there exists a quadratic function $V$ (resp. a piecewise algebraic function $V$) such that $\forall x \in \re^n, \forall A \in \{A_1,\ldots,A_m\},\, V(A x)\leq \rho V(x).$  \\

By a \rj{\emph{piecewise algebraic Lyapunov function}} $V$ (of degree $d$), following the nomenclature of \cite{LagariasWang_Finiteness_Conjecture}, we mean a positive definite, homogeneous, and continuous function (of degree $d$) whose unit sphere is included in the zero level set of a polynomial $p(x)$, with the property that $p(0)\neq 0.$ We refer to $p$ as the ``associated polynomial'' of $V$.

We remark that the class of piecewise algebraic functions contains all types of Lyapunov functions presented in our introduction section. For example, if $V$ is a pointwise minimum of polynomials $q_1,\ldots,q_m$, then the associated polynomial would be $p(x)=\prod_{i=1}^m (q_i(x)-1)$. (Note that the unit sphere of $V$ is the union of the unit spheres of the polynomials $q_i$.) Similarly, if $V$ is a pointwise maximum of polynomials $q_1,\ldots,q_m$, then the associated polynomial would be $p(x)=\sum_{i=1}^m (q_i(x)-1)^2$. (Note that in this case the unit sphere of $V$ is the intersection of the unit spheres of the polynomials $q_i$.)}

\rj{
\aaa{We now present the generalized result of interest}.  The idea of the proof follows that of \cite{LagariasWang_Finiteness_Conjecture}.  However, we present here a simplified version, neglecting the quantitative estimation of the upper bound.
\begin{theorem}\label{thm:sos.Lyap.then.finite}
Let $\{A_1,\ldots,A_m\}$ be a set of $n \times n$ matrices {of JSR equal to one}. If there exists a \aaa{piecewise algebraic Lyapunov function} $V(x)$  that satisfies \begin{equation}\label{eq-pwa}\forall i=1,\ldots,m,\ \forall x\in\re^n, V(x)-V(A_ix)\geq 0, \end{equation} then $\{A_1,\ldots,A_m\}$ satisfies the finiteness property. Moreover, the length of the optimal product is upper bounded by a quantity that depends on $n$, $m$, and $d$ only, where $d$ is the degree of the polynomial associated with $V(x).$
\end{theorem}

\begin{proof}
{\bf Step 1.  }We start by proving that if $\rho=1$ and $V$ satisfies \eqref{eq-pwa}, there exists a \aaa{point $x^*\in\mathbb{R}^n$, with $V(x^*)=1,$} together with a sequence of matrices $A_{i_1},A_{i_2},\dots,$ such that for all $t\in\n ,$ $\aaa{V(A_{i_t}\dots A_{i_1}x^*)=1}.$  
Clearly, by (\ref{eq-pwa}), for any $x$ and for any $i_{1},i_2,\dots, i_t,$ $V(A_{i_t}\dots A_{i_1}x)\leq V(x).$  Now, for any finite $t,$ there exists an $x_t,\ V(x_t)=1,$ and a product $A_{i_{t,t}}\dots A_{i_{t,1}}$ such that \aaa{$$V(A_{i_{t,t}}\dots A_{i_{t,1}}x_t) =1. $$} (Indeed, the contrary would imply that $\forall x, \forall i_1,\dots i_t,\  V(A_{i_{t}}\dots A_{i_{1}}x)<\gamma <1,$ which would  contradict $\rho=1$.)  In fact, by (\ref{eq-pwa}), the product  has the stronger property that $A_{i_{t,t'}}\dots A_{i_{t,1}}x_t =1 \ \forall t'\leq t.$  \\
Finally, by compactness of the set $\{x:V(x)=1\},$ and compactness of the set of words $i_1\dots i_t:t\in \n$ under the standard metric\footnote{Take the following distance between two (finite or infinite) words: $d(w_1,w_2)=1/(2^l),$ where $l$ is the index of the first character for which $w_1(l)\neq w_2(l).$  The set of (finite and infinite) words is compact under this metric.}, the sequence $(x_t,(i_1,\dots i_t))$ has a subsequence which converges to some $(x^*,(l_1,l_2,\dots))$ with $V(x^*)=1.$ By continuity of $V(x),$ we have $V(A_{l_t}\dots A_{l_1}x^*)=1$ for all $t\geq 1.$  

\aaa{The rest of the proof follows the reasoning }of \cite{LagariasWang_Finiteness_Conjecture}.  \footnote{Our version is however simplified for our purposes and does not effectively compute the upper bound in the statement of the theorem. }

{\bf Step 2.  }Now, from the sequence $l_1,l_2\dots$ obtained above we extract a new one $j_1,j_2\dots =l_sl_{s+1}\dots$ for some $s$ with the supplementary property that there are infinitely many $t_k\in \n$ such that $j_1j_2\dots j_{t_k} =j_{t_k+L_k}\dots j_{t_k+L_k+t_k-1} $ for some $L_k>0:$ that is, infinitely many prefices of the sequence turn out to reappear  later in the sequence.  There always exists such a sequence $j_1,j_2,\dots$ (see \cite[Lemma 3.2]{LagariasWang_Finiteness_Conjecture}).  

{\bf Step 3.  }Let us denote \begin{equation}\label{eq-st}S_t=\{x \in \re^n : V(x)\leq 1,\ p(A_{j_1}x)=0,\ \dots,\  p(A_{j_t}\dots A_{j_1} x)=0\}.\end{equation}  These sets are nonempty (as they contain $x^*$) and bounded away from zero (as $p(0)\neq 0$).  By a standard algebraic geometry argument, the sets $S_t$ can only take finitely many different values (see \cite[Lemma 3.3]{LagariasWang_Finiteness_Conjecture}). They also are such that \begin{equation}\label{eq-st-incl}\forall t, S_{t+1}\subseteq S_t.\end{equation}   (All the inclusions are non-strict here.) Thus, there exists some $t_k$ (as defined in step 2 above) such that $S_{t_k} = S_{2t_k+L_k-1}.$   Plugging this last relation in \eqref{eq-st}, \eqref{eq-st-incl}, we obtain $$A_{t_k+L_k-1}\dots A_{1} S_{2t_k+L_k-1} \subseteq  S_{t_k}=S_{2t_k+L_k-1}. $$

{\bf Final Step.  }
Thus, we have a compact set $S,$ bounded away from zero, together with a product $A=A_{t_k+L_k-1}\dots A_{1}$ such that $AS\subseteq S.$ As a consequence, for any integer $T$, $(A)^T S\subset  S.$  Since $S $ is compact and bounded away from zero, $\rho(A)=1,$ and the proof is done.  

It is possible to effectively bound the length $t_k+L_k$ of the optimal product in the worst case; see \cite{LagariasWang_Finiteness_Conjecture} for details.
\end{proof}}
\rj{We remark that the theorem above positively answers \aaa{ a previous conjecture of ours (see \cite{ahmadi-jungers-ifac-2014}} after Theorem 6).}
%
%
We will need a quick lemma:

\begin{lemma}\label{lem-prop-blondel}
 The matrices in $$\cM= \{A_1,A_2\}=\left \{ \begin{bmatrix} 1 & 1 \\ 0 & 1\end{bmatrix},\begin{bmatrix} 1 & 0 \\ 1 & 1\end{bmatrix}\right \}$$ satisfy: \begin{equation} \label{prop-blondel}\forall x \neq 0,\  |\{Ax/||Ax||: \, A\in \cM^*\}| =\infty, \end{equation} where $\cm^*$ is the set of all products of matrices in $\cM.$
\end{lemma}
\begin{proof}
One can verify that the product $A=A_1A_2$ has two real eigenvectors with eigenvalues of different modulus. Thus, if $x$ is not an eigenvector of $A$ it is clear that $|\{A^tx/||A^tx||:t\in \n\}|=\infty.$ Now, if $x$ is an eigenvector of $A,$ one can check that $A_1x$ is not, and we can reiterate the argument on the set $\{A^tA_1x/||A^tA_1x||:t\in \n\}$.
\end{proof}
\aaa{We can now state the results that we were after.}

\begin{corollary}
Let $\cM_\alpha$ be the set containing the following two matrices  \begin{equation} \label{matrices-blondel}A_1=\begin{bmatrix} 1 & 1 \\ 0 & 1\end{bmatrix}/\rho, \ A_2=\alpha \begin{bmatrix} 1 & 0 \\ 1 & 1\end{bmatrix}/\rho,\end{equation} with $$\rho=\rho(\begin{bmatrix} 1 & 1 \\ 0 & 1\end{bmatrix},\alpha_* \begin{bmatrix} 1 & 0 \\ 1 & 1\end{bmatrix}),$$ $$ \alpha_*\simeq 0.749326546330367557943961948091344672091...\footnote{See~\cite{FP_explicit_counterexample} for an expression for the exact value of $\alpha_*.$ It is known that $\cM_{\alpha_*}$ violates the finiteness property.}.$$ For any positive integer $d$, there exists a positive number $\alpha<\alpha_*$ such that the set of matrices $\cM_\alpha$ is asymptotically stable under arbitrary switching but does not admit a polynomial Lyapunov function of degree $\leq d$.
\end{corollary}

\begin{proof}
Suppose by contradiction that each of these values of $\alpha$ admits a degree $d$ or smaller positive homogeneous polynomial Lyapunov function $V_\alpha$.  We suppose on top of that that the maximal coefficient of these polynomials $V_\alpha$ is one (scaling the coefficients if needed). Thus, the functions $V_\alpha$ tend to a limit $V= \lim_{\alpha \rightarrow \alpha^*} V_\alpha,$ which is again a homogeneous polynomial function with coefficients smaller or equal to one and degree smaller or equal to $d$ because the set of such functions is a compact set. Moreover, it is easy to see that this limit function is a valid Lyapunov function for every $\alpha<\alpha_*.$ Thus, $\cM_{\alpha_*}$ must satisfy the nonstrict inequalities:
\begin{equation}\label{eq-lyap-alpha}\forall A \in \cM_{\alpha_*}, V(Ax)\leq V(x). \end{equation}

Before we conclude, it remains to show that $V$ is a valid piecewise algebraic Lyapunov function as defined above.  The only nontrivial part is that $V$ is positive definite, i.e., that there is no $x\in \re^n,\ ||x||=1,$ such that $V(x)=0.$  In order to prove that, we make use of Lemma \ref{lem-prop-blondel}.  The lemma implies that if $V$ has a certain x, $||x||=1,$ such that $V(x)=0,$ by Equation \eqref{eq-lyap-alpha}, $V$ has infinitely many zeros on the Euclidean sphere, and thus it cannot be a homogeneous polynomial of degree at most $d$ with maximum coefficient equal to one.

We can thus apply Theorem \ref{thm:sos.Lyap.then.finite}, which implies that $\cM_{\alpha_*}$ satisfies the finiteness property, a contradiction.
\end{proof}

The next corollary is very similar but the matrix family that it presents is completely explicit.

\begin{corollary}
Consider the matrix family \\$(1-\frac{1}{k})\{A_1,A_2\}$, with 
$$A_1=\alpha^k\begin{bmatrix} 0 & 0 \\ 1 & 0\end{bmatrix},\ A_2=\alpha^{-1}\begin{bmatrix}
\cos\frac{\pi}{2k} & \sin\frac{\pi}{2k}  \\ -\sin\frac{\pi}{2k} &\cos\frac{\pi}{2k}
\end{bmatrix},$$ where $$1<\alpha<(\cos\frac{\pi}{2k})^{-1}.$$ For any positive integer $d$, there exists a positive integer $k$ such that the set of matrices $(1-\frac{1}{k})\{A_1,A_2\}$ is asymptotically stable under arbitrary switching but does not admit a polynomial Lyapunov function of degree $\leq d$.
\end{corollary}

\begin{proof} Consider the matrix family $\{A_1(k),A_2(k)\}$. \rj{It is shown in}\cite{LagariasWang_Finiteness_Conjecture} that the JSR of this family is $1$ for all $k$. On the other hand, the spectral radii of all products of length $\leq k$ are less than one, whereas there is a product of length $k+1$ which achieves the JSR; i.e., has spectral radius one. As a result, by increasing $k$, the length of the optimal product blows up. Hence, by \rj{Theorem~\ref{thm:sos.Lyap.then.finite},} the degree of a polynomial Lyapunov function (applied to the asymptotically stable family $(1-\frac{1}{k})\{A_1(k),A_2(k)\}$) must blow up as well.
\end{proof}
\aaa{In Section~\ref{sec:numerical.example}, we test the result of the above corollary by searching for polynomial Lyapunov functions using semidefinite programming.}

\section{Complexity of Lyapunov functions and undecidability}\label{section-undecidability}
In this section, we show that our statements on lack of upper bounds on
complexity of Lyapunov functions also follow in a straightforward manner from
undecidability results. Compared to the results of the previous sections, the
new statements are weaker in some sense and stronger in some other. They are
weaker in that the statements are \emph{non-constructive}. However, they imply
the stronger statement that the complexity of Lyapunov functions (e.g., degree
or number of pieces) cannot be upper bounded, not only as a function of $n$ and
$m$, but also as a computable function of $n$, $m$, \rj{\emph{and the entries of the
matrices in $\Sigma$}} (Corollary~\ref{cor:no.computable.bound}).  In addition to this, we can further
establish that the same statements are true for very simple and restricted
classes of matrices whose entries take two different values only (see Theorem \ref{thm-binary}).

\begin{theorem}\label{thm-undecidability-nonconstructive}
For any positive integer $d,$ there are
families of matrices of size $47\times 47$ that are asymptotically stable under arbitrary switching but do not admit (i) a polynomial Lyapunov
function of degree $d,$ or (ii) a polytopic Lyapunov function
with $d$ facets, or (iii) a max-of-quadratics Lyapunov function with $d$ pieces, or (iv) a min-of-quadratics Lyapunov
function with $d$ pieces.\end{theorem}
  The main ingredient in the proof is the following undecidability theorem, which is stated in terms of the JSR of a set of matrices.
\begin{theorem}(\cite{blondel03undecidable,BlTi2})\label{thm-undecidability}
The problem of determining, given a set of matrices $\Sigma,$ if $\rho(\Sigma)\leq 1$ is Turing-undecidable.
This result remains true even if $\Sigma$ contains only two matrices with nonnegative rational entries of size $47\times 47.$
\end{theorem}
%
We now show that this result implies Theorem \ref{thm-undecidability-nonconstructive}.  The main ingredient is
Tarski's quantifier elimination theory, which gives a finite time procedure for checking certain quantified polynomial inequalities.  The rest is a technical transformation of the problem ``$\rho\leq 1$?'' to the existence of a degree $d$ polynomial Lyapunov function.

\begin{proof}[Proof of theorem \ref{thm-undecidability-nonconstructive}]
We suppose by contradiction that every AAS set of matrices of size $47\times 47$ has a polynomial Lyapunov function of degree at most $d,$ for some even natural number $d.$  We claim that this implies the algorithmic decidability of the question ``$\rho \leq 1?$''.  

Indeed, by homogeneity of the JSR, we have $$\rho(\Sigma)\leq 1 \Leftrightarrow \forall \epsilon\in(0,1),\ \rho((1-\epsilon)\Sigma)< 1.  $$
Now, by the hypothesis above, the last statement should be equivalent to the existence of a polynomial Lyapunov function of degree less than $d$ for $(1-\epsilon) \Sigma,$ and thus we can rephrase the question $\rho(\Sigma)\leq 1$ as follows.  (In what follows, $\mathcal{P}^+_{n,d}$ is the set of polynomials in $n$ variables and (even) degree $d$ \aaa{that are positive definite}. This set is semialgebraic. See Lemma \ref{lem:nonnegative.polys.semialgebraic} for a similar statement.)
\begin{eqnarray}\nonumber \forall  \epsilon\in(0,1), \, \exists p(\cdot)\in \mathcal{P}^+_{n,d} \mbox{ such that }\\ \nonumber \forall x\in \re^{47}\setminus \{0\}, \ \forall A \in \Sigma,\quad p((1-\epsilon)Ax)<p(x). \end{eqnarray}

\aaa{The above assertion is algorithmically decidable via Tarski's quantifier elimination theory~\cite{Tarski_quantifier_elim}}. Thus, it allows to decide whether $\rho(\Sigma)\leq 1,$ contradicting Theorem \ref{thm-undecidability}.

The proof for the polytopic Lyapunov function and max/min-of-quadratics goes exactly the same way, by noticing that once their number of components is fixed, one can decide their existence by \aaa{quantifier elimination} as well.
\end{proof}

In \cite{ BlTi2}, the authors note that Theorem \ref{thm-undecidability} implies the following result:
\begin{corollary}(\cite{BlTi2})
There is
no effectively computable function\footnote{See (\cite{BlTi2}) for a definition.} $t(\Sigma),$ which takes an arbitrary set of matrices with rational entries $\Sigma,$ and returns in finite time a natural number \rj{
such that $$\rho(\Sigma)=\max_{t'\leq t(\Sigma)}{\max_{A\in \Sigma}{\rho(A)}}.$$}
\end{corollary}
The same corollary can be derived concerning the degree of a Lyapunov function.
\begin{corollary}\label{cor:no.computable.bound}
There is
no effectively computable function $d(\Sigma),$ which takes an arbitrary set of matrices with rational entries $\Sigma,$ and returns in finite time a natural number
such that if $\rho(\Sigma)<1,$ there exists a polynomial Lyapunov function of degree less than $d.$
\end{corollary}

Next, we show a similar result, which does not focus on the fixed size of the matrices in the family, but somehow on the complexity of the real numbers defining the entries of the matrices.  Namely, we show that such negative results also hold essentially for sets of binary matrices (that is, matrices with only 0/1 entries).  In fact, the very question $\rho \leq 1$ is easy to answer in this case (see \cite{jungersprotasovblondel06}), so, one cannot hope to have strong negative results stated in terms of binary matrices.  However, it turns out that for an arbitrary integer $K$ the question $\rho \leq K$ for binary matrices is as hard as the question $\rho \leq 1$ for rational matrices.  More precisely, we have the following theorem:

\begin{theorem}(\cite{jungers-blondel-finiteness}\label{thm-rational-binary})
Given a set of $m$ nonnegative rational matrices $\Sigma,$ it is possible to build a set of $m$ binary matrices $\Sigma'$ (possibly of larger dimension), together with a natural number $K$ such that for any product $A=A_{i_1}\dots A_{i_t}\in \Sigma^t,$ the corresponding product $A'_{i_1}\dots A'_{i_t}\in \Sigma'^t$ has numerical values in \rj{each of} its entries that are \rj{either} exactly equal to zero, or to \rj{some entry} in the product $A$ multiplied by $K^t.$  Moreover, for any entry in the product $A,$ there is an entry in the product $A'$ with the same value multiplied by $K^t.$
\end{theorem}
\rj{Thus, the above theorem provides a way to encode any set of rational matrices in a set of binary matrices (in the sense that (i) the joint spectral radius is the same, up to a known factor $K$, and (ii) that there is a perfect correspondence between products of matrices in $\Sigma$ and in $\Sigma'$).}
Theorem \ref{thm-undecidability} together with Theorem \ref{thm-rational-binary} allows us to prove another negative result on the degree of Lyapunov functions restricted to matrices with entries all equal to a same number $1/K,$ $K\in \q.$ Remark that the fact that the parameter $K\in \q$ has unbounded denominator and numerator is unavoidable in such an undecidability theorem, since for bounded values, there is a finite number of matrices with all entries in the range, and this \rj{would rule} out a result as the one in the theorem below.
\rj{\begin{theorem}\label{thm-binary}
There is no computable function $d:\n \times \q \rightarrow \n$ such that for any set of matrices of dimension $n$ with entries all taking values in the set  $\{0,K\}$ for some $K\in \q,$ the set is AAS if and only if there exists a (strict) polynomial Lyapunov function of degree $d(n,K).$\end{theorem}}
\begin{proof}
Given a set of rational matrices of size $47\times 47,$ Theorem \ref{thm-rational-binary} allows us to build a set of binary $n'\times n'$ matrices $\Sigma'$ such that $\rho(\Sigma')=K\rho(\Sigma).$ Thus, $\rho(\Sigma'/K)=\rho(\Sigma),$ and the existence of strict polynomial Lyapunov functions for AAS $n'\times n'$ matrices would again imply decidability of the question $\rho\leq 1$ for $\Sigma'/K,$ which again contradicts Theorem \ref{thm-undecidability}.
\end{proof}

\rj{Roughly speaking, our last theorem shows that in higher dimensions,} ``bad'' families of matrices that necessitate arbitrarily complex Lyapunov functions can have very simple and structured entries.

\section{Numerical validation}\label{sec:numerical.example}

\aaa{In this section, we test the result of Theorem~\ref{thm:Lyaps.blow.up.families} through a numerical experiment. This will also give us a feel for how quickly the complexity parameter of the Lyapunov function can blow up in practice.

As a representative example, we choose the matrices of Lagarias and Wang from~\cite{LagariasWang_Finiteness_Conjecture}, parameterized by a single integer $k$. For each $k>0$, we know the matrices have JSR$<1$, but we try to find the smallest degree sum of squares Lyapunov function that can prove this. A nice feature of this example is that since the matrices are $2\times 2$, the Lyapunov functions and their increments will be bivariate forms. It is well-known that all nonnegative bivariate forms are sums of squares (see, e.g.,~\cite{Reznick}). Hence, the minimum degree of an sos Laypunov function is also exactly the minimum degree of a polynomial Lyapunov function. However, working with the sos formulation allows us to to compute this minimum degree using semidefinite programming. See, e.g.,~\cite{PhD:Parrilo},~\cite{Tutorial_CDC14_DSOSLyap} for an explanation of how this is done. Our SDP solver throughout the example is SeDuMi~\cite{sedumi}. It would be worthwhile to perform similar experiments with min/max-of-quadratics or polytopic Lyapunov functions, etc. 

The input matrices to our problem are the following

$$A_1=(1-\frac{1}{10^3 k})\alpha^k\begin{bmatrix} 0 & 0 \\ 1 & 0\end{bmatrix},\ A_2=(1-\frac{1}{10^3 k})\alpha^{-1}\begin{bmatrix}
\cos\frac{\pi}{2k} & \sin\frac{\pi}{2k}  \\ -\sin\frac{\pi}{2k} &\cos\frac{\pi}{2k}\end{bmatrix},$$ where $$\alpha=\frac{1}{2}(1+(\cos\frac{\pi}{2k})^{-1})
,$$ and $k$ is a parameter. 
Our results are as follows:
\begin{table}[h]
\begin{tabular}{l|llllll}
$k$ &2  &3  &4  &5  &6  &7  \\ \hline
 $d$& 6 &10  &14  &16  &20  &24 
\end{tabular}
\end{table}

This table is stating, for $k=7$ for example, that the matrices are stable under arbitrary switching and yet this fact cannot be proven with a polynomial Lyapunov function of degree less than 24. This trend keeps growing as expected until we run in numerical issues with accuracy of the SDP solver. We believe this example can now serve as a parametric benchmark problem for comparing the power of optimization-based techniques for proving stability of switched systems.
}

\section{Conclusion}\label{sec:conclusion}

In this paper, we leveraged results related to non-algebraicity, undecidability, and the finiteness property of the joint spectral radius to demonstrate that commonly used Lyapunov functions for switched linear systems can be arbitrarily complex, even in fixed dimension, or for matrices with lots of structure.
%
%
%

If these negative results are bad news for the practitioner, it is worth mentioning that in practice the different Lyapunov functions often have complementary performance. So while there certainly exist instances which make all methods fail (as we have shown), one can hope that in practice, at least one of the different Lyapunov methods would be able to certify stability. In light of this, we believe it is important to (i) understand systematically how the different methods compare to each other, and (ii) identify subclasses of matrices that if stable, are guaranteed to admit ``simple'' Lyapunov functions. While the latter research objective has been reasonably achieved for quadratic Lyapunov functions, results of similar nature are lacking for even slightly more complicated Lyapunov functions (say, polynomials of degree $4$, or piecewise quadratics with $2$ pieces).

%

\bibliographystyle{plain}
\include{pablo_amirali.bib}
\bibliography{pablo_amirali}

\end{document}